\documentclass[10pt]{amsart}
\usepackage{amsmath,amssymb,amsthm,graphicx}

\newtheorem{theorem}{Theorem}    

\newtheorem{corollary}{Corollary}
\theoremstyle{definition}

\newtheorem*{remark*}{Remark}
\newcommand{\Z}{\mathbb{Z}}
\newcommand{\R}{\mathbb{R}}

\newcommand{\Q}{\mathbb{Q}}
\DeclareMathOperator{\supp}{supp}
\DeclareMathOperator{\FDTC}{FDTC}

\title{On a structure of random open books and closed braids}
\author{Tetsuya Ito}
\address{Research Institute for Mathematical Sciences, Kyoto university
Kyoto, 606-8502, Japan}
\email{tetitoh@kurims.kyoto-u.ac.jp}
\urladdr{http://www.kurims.kyoto-u.ac.jp/~tetitoh/}
\thanks{T.I. was partially supported by JSPS Grant-in-Aid for Young Scientists (B) 15K17540.}
 
\begin{document}

\begin{abstract}
A result of Malyutin shows that a random walk on the mapping class group gives rise to an element whose fractional Dehn twist coefficient is large or small enough. We show that this leads to several properties of random 3-manifolds and links. For example, random closed braids and open books are hyperbolic. 
\end{abstract}

\maketitle

Let $G$ be the mapping class group or the braid group of an oriented compact surface $S$ with connected boundary. 
The \emph{Fractional Dehn Twist Coefficients} (\emph{FDTC}, in short) is a quasimorphism $\FDTC: G \rightarrow \Q$ which plays an important role in contact 3-manifolds and open book decompositions \cite{hkm1}. When one studies (contact) 3-manifolds by using open book decompositions, the FDTC often appears as a primary assumption of various theorems. Here a \emph{quasimorphism} on a group $G$ is a map $\rho:G \rightarrow \R$ which has the property $\sup_{g,h \in G} |\rho(gh)-\rho(g)-\rho(h)| \leq \infty$.

The aim of this paper is to point out observations that the results of a random walk on $G$, combined with the results on open book decompositions and closed braids, lead to various interesting and non-trivial properties on random open books and closed braids. Although these results are direct consequences of known results and their proofs, they illustrate a general picture of random 3-manifolds and links.

Let $\mu$ be a probability measure on $G$ with finite support, and let $H_{\mu}$ be the sub semi-group of $G$ generated by $\supp(\mu)$, the support of $\mu$. Consider the simple random walk with respect to $\mu$ starting from $h \in G$. Let $g_{k}$ be the random variable representing the position of the random walk at time $k$. The theory of random walks tells that $g_{k}$ has the following properties.

\begin{enumerate}
\item If $\mu$ is non-elementary, which we mean that $H_{\mu}$ contains pseudo-Anosov elements with distinct fixed points on the Thurston boundary of Teichm\"uller space, then the probability that $g_{k}$ is pseudo-Anosov goes to one as $k \to \infty$ \cite{mah},\cite[Corollary 0.6]{mal}.
\item For a quasimorphism $\rho: G \rightarrow \R$, if $\mu$ is unbounded with respect to $\rho$, which we mean that $\rho(H_{\mu}) \subset \R$ is unbounded, then for any $C \geq 0$, the probability that $|\rho (g_{k})| \leq C$ goes to zero as $k \to \infty$ \cite[Corollary 0.2]{mal}.
\end{enumerate}

In the rest of the paper, except Theorem \ref{theorem:5}, we always assume that a probability measure $\mu$ is non-elementary and unbounded with respect to the FDTC quasimorphism. Thus, a random element $g_{k}$ is pseudo-Anosov with sufficiently large or small FDTC, with asymptotic probability one.

The first result justifies our naive expectation for ``generic'' 3-manifolds --  one can expect a random 3-manifold admits various nice structures.
For an open book $(S,g)$ we denote the corresponding 3-manifold by $M_{(S,g)}$.

\begin{theorem} 
\label{theorem:random3m}
As $k \to \infty$, the probability that an open book $(S,g_{k})$ has the following properties goes to one. 
\begin{enumerate}
\item[(a)] $M_{(S, g_{k})}$ is hyperbolic. (In particular, $M_{(S, g_{k})}$ is irreducible and atoroidal.)
\item[(b)] For a fixed $C>0$, $M_{(S,g_{k})}$ contains no incompressible surface of genus less than $C$.
\item[(c)] Either $(S,g_{k})$ or $(S,g_{k}^{-1})$ supports a weakly symplectically fillable and universally tight contact structure, which is a perturbation of a co-oriented taut foliation. (In particular, $M_{(S,g_{k})}$ admits a co-oriented taut foliation).
\item[(d)] $M_{(S,g_{k})}$ is not a Heegaard-Floer L-space.
\end{enumerate}
\end{theorem}
\begin{proof}
(a) follows from \cite[Theorem 8.3]{ik}: $M_{(S,g)}$ is hyperbolic if $g$ is pseudo-Anosov with $|\text{FDTC}(g)|>1$. (b) follows from \cite[Theorem 7.2]{ik}: an existence of incompressible surface of genus $C>0$ implies $|\text{FDTC}(g)| \leq C$.
(c) follows from \cite[Theorem 1.2]{hkm2}: $(S,g)$ supports a desired contact structure if $g$ is pseudo-Anosov with $\text{FDTC}(g)>1$. (c) and \cite[Theorem 1.4]{os}, that asserts that an L-space does not admit a co-oriented taut foliation, prove (d).
\end{proof}

Note that when we take the starting point $h$ and the measure $\mu$ so that $M_{(S,h)}$ is an integral homology sphere and that $\supp(\mu)$ is contained in the Torelli group, then $M_{(S,g_{k})}$ is always an integral homology sphere. Thus the above theorem says that a random integral homology sphere also has various properties as we listed in Theorem \ref{theorem:random3m}. 
Since the fundamental group of integral homology sphere admitting co-oriented taut foliation is left-ordrdable \cite{cd}, we have the following result, which gives a supporting evidence of \cite[Conjecture 1]{bgw}. 
  
\begin{corollary} 
As $k \to \infty$, the probability that a random open book representing an integeral homology sphere $M_{(S,g_{k})}$ has the following properties goes to one.
\begin{enumerate}
\item $M_{(S, g_{k})}$ is not an L-space.
\item $M_{(S, g_{k})}$ admits a co-oriented taut foliation.
\item $\pi_{1}(M_{(S,g_{k})})$ is left-orderable.
\end{enumerate}
\end{corollary}

Next we study a random link in 3-manifold. Fix a 3-manifold $M$ and its open book decomposition $(S,\phi)$.
Then the closure of $g \in B_{n}(S)$, an $n$-strand braid of $S$, yields an oriented link in $M_{(S,\phi)}$ which we denote by $\widehat{g}$.

We regard $g$ and $\phi$ as an element of $MCG(S-\{n \text{ points}\})$ and consider their product $g \phi$. The FDTC of a (closed) braid $\widehat{g}$ is defined as the FDTC of $g \phi$, viewed as an element of $MCG(S-\{n \text{ points}\})$.  Here we remark that $g \phi$ is nothing but the monodromy of the  fibration $M_{(S,\phi)}\setminus (B \cup \widehat{g}) \rightarrow S^{1}$ that comes from the open book fibration $M_{(S,\phi)}\setminus B \rightarrow S^{1}$, where $B$ denotes the binding. (See \cite[Section 4]{ik} for details). 

The first part of the next result generalizes \cite{ma}.

\begin{theorem}
As $k \to \infty$, the probability that $\widehat{g_{k}}$ is a hyperbolic link in $M_{(S,\phi)}$ goes to one as $k \to \infty$.

Moreover, if $\widehat{g_{k}}$ is null-homologous (for example, when $M_{(S,\phi)}$ is an integral homology sphere), then for any fixed constant $C>0$, the probability that $g(\widehat{g_{k}}) \leq C$ goes to zero as $k \to \infty$. Here $g(\widehat{g_{k}})$ denotes the genus of $\widehat{g_k}$.
\end{theorem}
\begin{proof}
This follows from \cite[Theorem 8.4, Corollary 7.13]{ik}: $\widehat{g_{k}}$ is hyperbolic if $g_{k}\phi$ is pseudo-Anosov with $|\text{FDTC}(g_{k}\phi)|>1$, and that $|\FDTC(g_{k}\phi)|$ gives an lower bound of $g(\widehat{g_{k}})$.
\end{proof}

Then we analyse more precise structures on a random usual closed braid in $S^{3}$. 
First we point out a surprising consequence.

\begin{theorem}
\label{cor:3}

The probability that two random braids $\alpha_{k}, \beta_{l} \in B_{n}$ are non-conjugate but represent the same link goes to zero as $k,l \to \infty$. In other words, the probability that $\alpha_k$ and $\beta_l$ are conjugate, given that they represent the same link, tends to one as $k, l \to \infty$.
\end{theorem}
\begin{proof}
This follows from \cite[Theorem 2.8]{it}, based on a deep result of Birman-Menasco \cite{bm}: There is a constant $r(n)$ such that for $n$-braids $\alpha,\beta$ with $|\text{FDTC}|>r(n)$ the closures of $\alpha$ and $\beta$ are the same if and only if they are conjugate.
\end{proof}

Note that this also says that the closures of two random braids are \emph{transverse} isotopic if they are \emph{topologically} isotopic. Thus, a random closed braid model of random transverse links are the same as a random closed braid model of random topological links.

Next we address a question concerning the transient properties. For $g\in \Z_{\geq 0}$, let $S(n,g)$ be the subset of the braid group $B_{n}$ consisting of a braid whose closure represents a link of genus $\leq g$. The following result was conjectured in \cite{mal}.

\begin{theorem}
\label{cor:4}
$S(n,g)$ is transient for the random walk $\{g_k\}$ on $B_{n}$. 
\end{theorem}
\begin{proof}
In the proof of \cite[Theorem 1.2]{it2}, it is shown that for $\beta \in B_{n}$, if $g(\widehat{\beta}) \leq g$ then $\beta$ is conjugate to a braid represented by a word $W$ over the standard generator $\{\sigma_{1}^{\pm 1},\ldots, \sigma_{n-1}^{\pm1}\}$ such that the number of $\sigma_{1}^{\pm 1}$ in $W$ is at most $2g$. This shows that such a braid $\beta$ is written as a product of at most $4g$ reducible braids. Let $T_n \subset B_{n}$ be the set of all non pseudo-Anosov $n$-braids. Then $S(n,g) \subset T_{n}^{4g}$. 
By \cite[Corollary 0.7]{mal}, $T_{n}^{4g}$ is transient for the random walk $\{g_k \}$ hence so is $S(n,g)$.
\end{proof}

Finally we point out another application of quasi-morphism arguments.
A knot concordance invariant $c$ having certain properties, like signature $\sigma$, Rasmussen's invariant $s$ \cite{ra}, the $\tau$ invariant of knot Floer homology \cite{os0}, gives rise to a quasi-morphism $\rho_c:B_{n} \rightarrow \Z$, by defining $\rho_c(\alpha) = c(\widehat{\alpha})$ \cite{bra}.

\begin{theorem}
\label{theorem:5}
{$ $}
\begin{enumerate}
\item If $\mu$ is unbounded with respect to $\rho_{\sigma}$ or $\rho_{s}$, then  the probability that $\widehat{\alpha_{k}}$ is not slice (more strongly, for a fixed constant $C$, the probability that the slice genus  $g_{4}(\widehat{\alpha_k})$ is greater than $C$), goes to one as $k \to \infty$. 
\item If $\mu$ is unbounded with respect to $\rho_{s}-\rho_{\sigma}$, the probability that $\widehat{\alpha_k}$ is non-alternating goes to one as $k \to \infty$.
\end{enumerate} 
\end{theorem}
In Theorem \ref{theorem:5}, we can use $\tau$ invariant instead of Rasmussen's invariant $s$.
\begin{proof}
The first assertion follows from the fact that $|\sigma(\widehat{\alpha})|,|s(\widehat{\alpha})| \leq 2g_{4}(\widehat{\alpha})$. The second assertion follows from the fact that $\sigma(\widehat{\alpha})=s(\widehat{\alpha})$ if $\widehat{\alpha}$ is alternating \cite[Theorem 3]{ra}.
\end{proof}

\begin{remark*}

The assumption that $\mu$ is unbounded with respect to $\phi_{s}-\phi_{\sigma}$ is satisfied, for example, if $n>2$ and some powers of $\Delta^{4} =(\sigma_{1}\cdots\sigma_{n-1})^{2n}$ belongs to $H_{\mu}$. This follows from the following facts. 
\begin{itemize}
\item $|\sigma(\widehat{\alpha \Delta^{4}})| \leq |\sigma(\widehat{\alpha})| + n^{2} - 1$ \cite[Lemma 4.1]{gg}.
\item If $\alpha$ is a positive braid, $|s(\widehat{\alpha \Delta^{4}})| = |s(\widehat{\alpha})| + 2n(n-1)$, because the closure $\widehat{\alpha}$ of a positive $n$-braid satisfies $2g_{4}(\widehat{\alpha})= |s(\widehat{\alpha})| = e(\alpha)-n+1$ where $e(\alpha)$ is the exponent sum of the braid $\alpha$ \cite[Theorem 4]{ra}.
\item For any $\beta \in B_{n}$, $\beta\Delta^{4N}$ is a positive braid for sufficiently large $N>0$.
\end{itemize}

\end{remark*}

\end{document}